\documentclass{article}
\usepackage{amsmath,amsfonts,latexsym}
\usepackage[english]{babel}
\usepackage[pdftex]{graphicx}
\usepackage{hyperref}
\usepackage{multicol}
\usepackage{amsmath}
\usepackage{amsthm}
\usepackage{amssymb}
\usepackage{amsfonts}
\usepackage{multicol}
\usepackage{color}
\usepackage{authblk}

\usepackage[all]{xy}

\usepackage{tikz}
    \usetikzlibrary{cd}
    \usetikzlibrary{shapes.misc}

\newtheorem*{theorem*}{Theorem}
\newtheorem{lemma}{Lemma}
\newtheorem{theorem}{Theorem}
\newtheorem{proposition}{Proposition}

\usepackage{hyperref}
\hypersetup{
colorlinks=true,       
    linkcolor=blue,          
    citecolor=blue,        
    filecolor=blue,      
    urlcolor=blue           
}

\begin{document}
 
\title{A brute force computer aided proof of an existence result about extremal hyperbolic surfaces}

\author{Ernesto Girondo\thanks{Partially supported by grant MTM2016-79497-P} and Cristian Reyes\thanks{This work was funded by the CONICYT PFCHA/Becas Chile 72180175}}

\date{Departamento de Matem\'aticas, Universidad Aut\'onoma de Madrid (Spain) \\
\ \\
ernesto.girondo@uam.es, cristianr.reyes@estudiante.uam.es}

\maketitle

\begin{abstract}
Extremal compact hyperbolic surfaces contain a packing of discs of the largest possible radius permitted by the topology of the surface. It is well known that arithmetic conditions on the uniformizing group are necessary for the existence  of a second extremal packing in the same surface, but constructing explicit examples of this phenomenon is a complicated task.
We present a brute force computational procedure that can be used to produce examples in all cases.
\end{abstract}

\section{Introduction}

It is well known  that the radius of a metric disc embedded in a compact hyperbolic surface cannot be larger than certain bound depending only in the topology of the surface (cf. \cite{Bavard_1996}, \cite{Girondo_Gonzalez-Diez_1999},  \cite{Girondo_Nakamura_2007}). Recently, this result has been extended to embeddings of  $k$-packings (i.e.  collections of a number $k$ of pairwise disjoint metric discs of given radius) in hyperbolic surfaces, first in the orientable case (see \cite{Girondo_2018}) and then in the non-orientable case (\cite{preprint}). By the term \emph{extremal surface} we mean a (orientable or non-orientable) surface $S$ with an embedded $k$-packings realizing the upper bound for the radius, given by  
$$\cosh R = \frac{1}{2\sin \frac{k\pi}{6(k-\chi)} }.$$

\

Extremal surfaces may be studied from the point of view of discrete groups of M\"obius transformations. We can regard such a surface $S$ as
a quotient $S \simeq \mathbb{D} / K$ of the universal covering space $\mathbb{D}$, the unit disc, by a \emph{uniformizing group} $K$, which is a Fuchsian group or a proper NEC group depending on whether $S$ is orientable or not.

\

It has been observed that extremality translates into algebraic, and even arithmetic properties of the group $K$. First, a necessary and sufficient condition for the existence of an extremal $k$-packing in $S$ is that  the group $K$ must be contained in a (Fuchsian or NEC) triangle group $\Delta=\Delta(2,3,N)$ where $N=6(k-\chi)/k$, the index being  $6(k-\chi)$ or $12(k-\chi)$ for the orientable and the non-orientable case, respectively (\cite{Girondo_2018}, \cite{preprint}). It follows that topology does not impose any restriction to extremality for $k=1,2,3$ and $6$. On the other hand,  when $k$ does not divide $6\chi$, embedded $k$-packings  cannot be  as dense as they are in hyperbolic space, and therefore extremal $k$-packings do not exist for such topological classes of hyperbolic surfaces.

\

Actual existence of extremal surfaces for the admissible pairs $(k,\chi)$ (or, equivalently $(k, g)$, where $g$ stands for the orientable or non-orientable genus) has been shown either as a consequence of known results about inclusions of Fuchsian groups (\cite{Girondo_2018}, where a well-known theorem by Edmonds et al. \cite{Edmonds_Ewing_Kulkarni_1982} is used) or as a result of a clever combinatorial procedure called \emph{edge-grafting} (\cite{preprint}).

\

There is a natural question where, surprisingly,  arithmetic seems to play a role related to extremality. Asking if a extremal surface of genus $g$ may contain several distinct extremal $k$-packings translates into asking if the uniformizing group $K$ can be contained in two different (conjugate) copies of the corresponding triangle group $\Delta$. A combination of known results by Margulis, Singerman and Takeuchi shows that extremal $k$-packings are often necessarily unique within their surfaces, the possible exceptions corresponding to values $(k, g)$ such that the triangle group $\Delta$ is arithmetic (see \cite{Takeuchi_1977_a} for the definition of arithmeticity). This way, first  in \cite{Girondo_2018} for the orientable case and later in \cite{preprint} for the non-orientable case, it was proved the following

\begin{theorem*} If $X$ is a compact hyperbolic  $k$-extremal surface of genus $g$ and 
$$
N:=\frac{6(k-\chi)}{k} \not\in\{7,8,9,10,11,12,14,16,18,24,30\}
$$ then the extremal $k$-packing is unique in $X$.
\end{theorem*}

\

Summarizing, a combination of topology/combinatorics and arithmetics restricts the possible pairs $(k,g)$ for which there could be an extremal $k$-surface of genus $g$ with several extremal $k$-packings to values such that

\begin{itemize}
\item Orientable case: $k$ divides $12g+6k-12$ and $N=\frac{12g+6k-12}{k} \in\{7,8,9,10,11,12,14,16,18,24,30\}$
\item Non orientable case: $k$ divides $6g+6k-12$ and  $N=\frac{6g+6k-12}{k}  \in\{7,8,9,10,11,12,14,16,18,24,30\}$.
\end{itemize}

Notice that these are necessary conditions. Showing that they are also sufficient is far from being obvious. In the first cases considered in the literature, namely those with $k=1$, see \cite{Girondo_Gonzalez-Diez_2002_2}, \cite{Girondo_Nakamura_2007}) the strategy employed was as follows. First, determine all possible extremal surfaces, something that can be done by finding a representative of all conjugacy classes of Fuchsian/proper NEC groups of the appropriate index inside the triangle group $\Delta$ with the help of some algebra package such as GAP. Then, perform an exhaustive metric study  (similar to what was done for instance in \cite{Girondo_Gonzalez-Diez_2002_2}; see Section \ref{sec:brute} below for some hints about what we mean with this) in order to determine if a second extremal $k$-packing exists in any of the surfaces.

\

The serious problem for extending this strategy to $k>1$ is that sometimes it is virtually impossible to list the huge number of subgroups $K<\Delta$ involved.

\section{The main result}

We shall prove  the following theorem:

\begin{theorem} \label{th_main} Let $g\ge 3$ and $k \ge 1$ be such that $k$ divides $6g+6k-12$. Then there exists a compact non-orientable $k$-extremal surface of genus $g$ with more than one extremal $k$-packing if and only if  
$$
N:=\frac{6g+6k-12}{k} \in\{7,8,9,10,11,12,14,16,18,24,30\}.
$$ 
In all other cases extremal $k$-packings are unique within compact non-orientable $k$-extremal surfaces of genus $g$.
\end{theorem} 

A direct consequence, given by a simple argument involving the orientable double cover of non-orientable, $k$-extremal surfaces (see the last section of \cite{preprint}), is

\begin{theorem} \label{th_main2} Let $g \ge 2 $ and $k \ge 1$ be such that $k$ divides $12g+6k-12$. Then there exists a compact Riemann surface of genus $g$ that is  $k$-extremal and has more than one extremal $k$-packing if and only if  
$$
N:=\frac{12g+6k-12}{k} \in\{7,8,9,10,11,12,14,16,18,24,30\}.
$$ 
In all other cases extremal $k$-packings are unique within compact non-orientable $k$-extremal surfaces of genus $g$.
\end{theorem}

Note that by the characterization of compact non-orientable $k$-extremal surfaces in terms of triangle groups (cf. \cite{preprint}), we know that if such a surface $S\simeq  \mathbb{H}\slash K$ admits two extremal $k$-packings, then we must have two inclusions $$K\le \Delta_1^{\pm}=\Delta_1^{\pm}\left(2,3,\frac{6g+6k-12}{k}\right),\quad  K\le \Delta_2^{\pm}=\Delta_2^{\pm}\left(2,3,\frac{6g+6k-12}{k}\right)$$ both with index $12g+12k-24$. 

As a consequence, a compact non-orientable  $k$-extremal surface whose NEC group $K$ is contained in a non-arithmetic extended triangle group $\Delta^{\pm}$, has a unique extremal $k$-packing (see \cite{preprint}). This happens whenever $$\frac{6g+6k-12}{k}\not\in\{7,8,9,10,11,12,14,16,18,24,30\}$$

\

The difficult statement is the converse, since there is apparently no theoretical reason ensuring the existence of the uniformizing group $K$ with the required properties for all the pairs $(k,g)$ involved. We are forced to find explicitly one example for every such pair $(k,g)$. 

\

Taking the terminology from  \cite{preprint}, for $N\in\{7,8,9,10,11,12,14,16,18,24,30\}$ a pair $(k, g)$ such that $\frac{6g+6k-12}{k}=N$ is called  \emph{primitive} if both $k$ and $g$ are minimal among all the pairs $(k,g)$ verifying the same relation with $N$. We denote $(k_N,g_N)$ the primitive pair for a given $N$. The existence of extremal non-orientable $k_N$-extremal surfaces of genus $g_N$ for the eleven primitive pairs (one for each of the relevant values of $N$) is enough to ensure the existence of $k$-extremal surfaces of genus $g$ for all the pairs $(k,g)$ in Theorem \ref{th_main}, see \cite{preprint}.

\

Notice that four of the eleven required surfaces can be found already in the literature, since they correspond to the cases for which $k_N=1$ (see \cite{Girondo_Nakamura_2007}, \cite{Nakamura_2009}, \cite{Nakamura_2012}, \cite{Nakamura_2013} and \cite{Nakamura_2016}). These surfaces  correspond to $N=12, 18, 24$ or $30$. 
For the cases $N=8,9,10$ it is still  possible to do an exhaustive analysis of all  compact non-orientable $k_N$-extremal of genus $g_N$ (similar to that one done in \cite{Girondo_Nakamura_2007}), since we can compute all proper NEC uniformizing subgroups of the extended triangle group $\Delta^{\pm}(2,3,N)$ with the required index. But the computation of all the corresponding subgroups of $\Delta^{\pm}(2,3,N)$ with $N=7, 11,14,16$ is a inadequate  approach, since there are sometimes hundreds of thousands of such surfaces.  We propose a different strategy, that is the brute-force procedure we refer to in the title of this paper, and we devote the rest of the paper to explain the details.

\section{The brute force construction} \label{sec:brute}

For a given $N$, any compact non-orientable $k_N$-extremal surface of genus $g_N$ is uniformized by a proper NEC subgroup $K$ of index $2k_NN$ inside a triangle group $\Delta^{\pm}(2,3,N)$. Such a group admits a fundamental domain $F$ consisting of a union of $k_N$ regular $N$-gons of angle $2\pi /3$, the extremal $k_N$-packing being simply induced by the discs inscribed to these polygons.

\begin{figure}[!htbp]
\begin{center}
\includegraphics[width=0.9\textwidth]{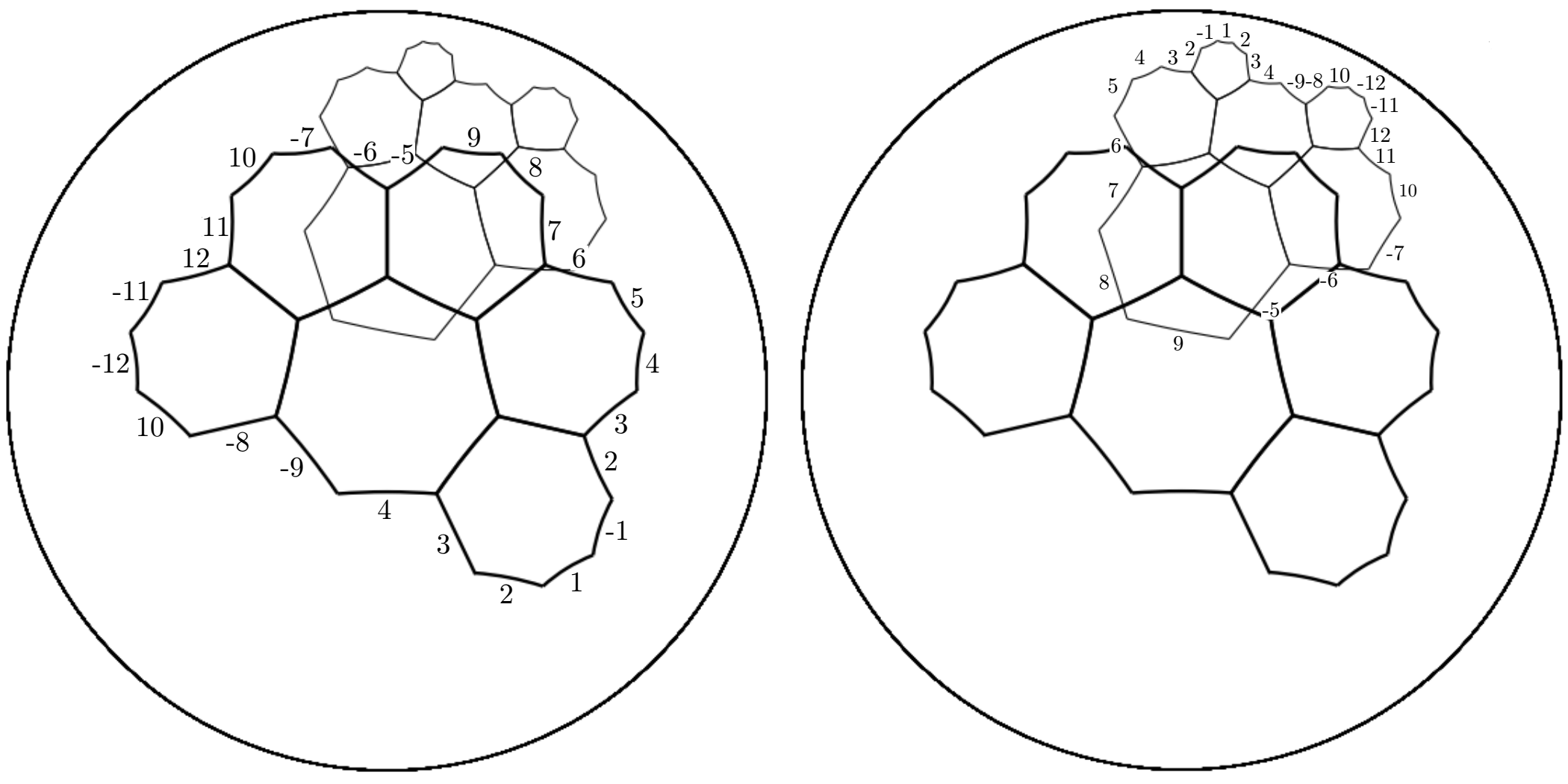} 
\caption{A primitive example with two different extremal $k$-packings for the case $N=7$. The figure shows two different fundamental domains for the same proper NEC uniformizing group $K$, both decomposed as the union of six regular heptagons of angle $2\pi /3$. The numbers indicate side-pairing transformations generating $K$, and a minus sign means that the corresponding transformation is orientation-reversing. Surprisingly, both sets of side pairings generate the same group $K$.}\label{N7}
\end{center}
\end{figure}

If the resulting extremal surface $S\simeq \mathbb{D}/K$ admits a second extremal $k_N$-packing, the group $K$ would have a second fundamental domain $F'$ isometric to $F$ but not $K$-related to it (Figure \ref{N7} shows an example of this phenomenon in the case $N=7$, where $k_7=6$ and $g_7=3$). The main remark here is that the $k_N$ centers of the $N$-gons forming $F'$ must be very special points, since they must  respect very strong metric conditions with respect to $K$. The following lemma, taken from \cite{Girondo_Nakamura_2007}, describes such conditions.

\begin{lemma}\label{lemadist}
Let $S$ be a compact non-orientable $k$-extremal surface of genus $g$ uniformized by $K$, and let $\pi:\mathbb{D}\longrightarrow \mathbb{D}/K \simeq S$ be the natural projection. Let $p=[z]_K$ be the center of one of the $k$ discs forming an extremal $k$-packing inside $S$, and denote $N=\frac{6g+6k-12}{k}$ and $\mathcal{T}$ a tessellation of $\mathbb{D}$ by regular $N$-gons of angle $2\pi / 3$. Then, for every $\gamma \in K$ the distance $d(z, \gamma(z))$ must agree with the distance between some pair of polygon centers of $\mathcal{T}$. We call these \emph{admissible distances}.
\end{lemma}

We need to compute as many admissible distances as we can, in order to have a precise idea about the possible displacement of centers of discs belonging to \emph{hidden} extremal $k_N$-packings. The following observation is useful for this purpose:

\begin{lemma} \label{le:dists}
 Consider a tessellation $\mathcal{T}$ of $\mathbb{D}$ given by copies of a regular $N$-gon $P_0$ of angle $\frac{2\pi}{3}$ centered at the origin $0$. Consider a triangle  $\triangle ABC$ of angles $\pi/3, \pi / N$ and $\pi / 2$ with vertices at a vertex $A$ of $P_0$, the origin $B=0$ and an edge midpoint $C$. Denote by $a,b,c$ the reflections with respect to the lines $\overline{BC},\overline{CA}$ and $\overline{AB}$ respectively, and let $R_{m}=(ca)^{m}ab$. 

Given any polygon $P$ in $\mathcal{T}$, there is a transformation of the form $R_{i_1}R_{i_2}\cdots R_{i_k}$ that sends $P_0$ to $P$, where  $0 \le i_j < N$ for $j=1, \ldots, k$.
\end{lemma}

\begin{proof}
If we select an arbitrary polygon $P$ in the tessellation, we can always connect the origin to its center by a hyperbolic polygonal curve $\Gamma$ joining centers of a sequence of pairwise adjacent polygons $P_0 , P_1, \ldots , P_k=P$. Figure \ref{14-7} illustrates the construction for an example in the case $N=7$. 

\

\begin{figure}[!htbp]
\begin{center}
\includegraphics[width=0.4\textwidth]{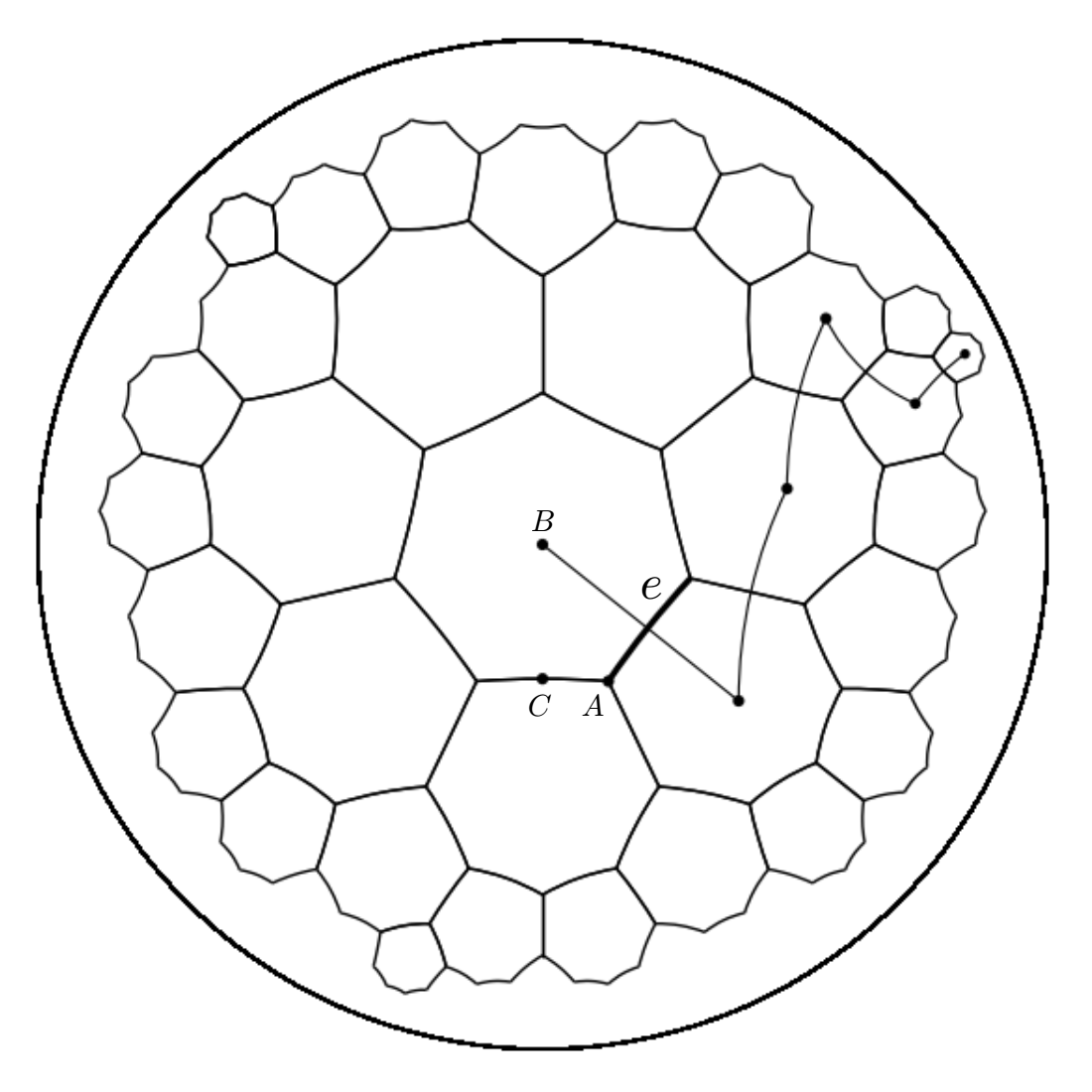} 
\caption{The tesselation $\mathcal{T}$ and a sketch of the construction for the case $N=7$. In this case $k=5$ and  $i_1=1, i_2=6, 
i_3=4, i_4=1, i_5=5$.}\label{14-7}
\end{center}
\end{figure} 

Now, label the edges of the central polygon counterclockwise, with the label $0$ at the edge containing the point $C$, and let $i_1$ be the label of the edge $e$ where $\Gamma$ cuts the central polygon $P_0$. We assign then the label $0$ to $e$ considered as an edge of the polygon $P_1$ meeting $P_0$ along $e$, and complete the labelling $1,2,\ldots$ of the remaining edges of $P_1$ counterclockwise. Next, define $i_2$ as the label of the edge of $P_1$ where $\Gamma$ leaves $P_1$, and so on. This way we construct a sequence of numbers $i_1, i_2, \ldots, i_k$. We claim that the transformation
$R_{i_1}R_{i_2}\cdots R_{i_k}$ sends $P_0$ to $P_k=P$.

\

The proof can be done by induction in the index $k$. Taking $k=1$ corresponds to the case when $P$ is a direct neighbor of the central polygon $P_0$, therefore this case is absolutely obvious.

Now, assume that the statement is true for polygons whose center can be joined to the origin by polygonal lines with $k$ segments. Assume $\Gamma$ is a polygonal line connecting the origin to the center of certain polygon $P$, crossing the sequence $P_0, P_1, \ldots , P_{k}, P_{k+1}=P$, and define $R_{i_1}R_{i_2}\cdots R_{i_k}R_{i_{k+1}}$ as above. By the induction hypothesis, the transformation $R=R_{i_1}R_{i_2}\cdots R_{i_k}$ sends $P_0$ to $P_k$, therefore $RR_{i_{k+1}}$ sends $P_0$ to one of the neighbors of $P_k$ that, by construction, is precisely 
 $P_{k+1}$.
\end{proof}

Note that the set of $N^k$ transformations $\mathcal{R}_k = \{R_{m_1}R_{m_2}\cdots R_{m_k} , \ 0\le m_i < N \}$ contains elements that move the central polygon $P_0$ to any polygon that can be connected to $P_0$ with a polygonal $\Gamma$ of $k$ steps. A collection of admissible distances of the tessellation $\mathcal{T}$ can be constructed by computing the distance $d(0, R(0))$ for every $R$ in $\mathcal{R}_k$. Larger values of $k$ yield larger subsets of the full set of admissible distances.

\

The locus of points which are moved certain prescribed distance by a hyperbolic transformation can be expressed in terms of distance to the so-called axis of the transformation (invariant geodesic), by the following well-known formula for the displacement function (cf. \cite{Beardon_1983}).

\begin{proposition}  If $g$ is a hyperbolic transformation with translation length $T$ and axis $A$, then $$\sinh \frac{d(z,g(z))}{2} =\cosh d(z,A) \sinh \frac{T}{2}.$$
\end{proposition}

One can easily prove a similar result for glide-reflections, for which we still call translation length the displacement of points belonging to the axis:

\begin{proposition} If $g$ is a glide-reflection with translation length $T$ and axis $A$, then $$\cosh \frac{d(z,g(z))}{2} =\cosh d(z,A) \cosh \frac{T}{2} .$$
\end{proposition}

\begin{proof}
Since any glide reflection is conjugate to a map of the form $h:z\mapsto -k\overline{z},\ k>1$, all we have to do is to prove the result for such transformations $h$. In order to check that the formula holds, we compute the three magnitudes involved using well-known expressions for the hyperbolic sine or cosine of the hyperbolic distance between points of the upper half-plane $\mathbb{H}$ (see \cite{Beardon_1983}, Theorem 7.2.1).

First, for the term on the left of the equation we have
$$\cosh \frac{d(z,h(z))}{2}=\frac{|z-\overline{h(z)}|}{2\sqrt{\mathrm{Im}(z)\mathrm{Im}(h(z))}}= \frac{|z||1+k|}{|y|2\sqrt{k}},$$
where $z=x+iy$. Now, since the axis $A'$ of $h$ is the imaginary axis, we can compute  
$$\cosh d(z,A')= \cosh d(z,i|z|)= 
1+\frac{|z-i|z||^2}{2 \mathrm{Im}(z) \mathrm{Im}(i|z|)}=\frac{|z|}{y}$$
and, finally, we obtain $T$ using that $z=i$ belongs to the axis, so that
$$\cosh \frac{T}{2} =\cosh \frac{d(i, ki)}{2} =\frac{|i+ki|}{2 \sqrt{\mathrm{Im}(i)} \sqrt{\mathrm{Im}(ki)}}=\frac{(1+k)}{2\sqrt{k}}.$$ 
\end{proof}

From the last two propositions we see that the locus of points with a prescribed displacement under a hyperbolic isometry or a glide reflection coincides with the locus of points at certain distance from the axis of the transformation. This set has two components, one at each side of the axis, which are arcs of generalized circles (i.e. arcs of circumference or of straight lines), see \cite{Beardon_1983}, Section 7.20. We call such sets \emph{bananas} for obvious reasons (see Figure \ref{14-3}).

\

We have now all the ingredients needed for our brute force project. Assume we want to decide if there exists a compact hyperbolic $k$-extremal surface $S$ of genus $g$ with more than one extremal $k$-packing.  The plan is as follows:
\begin{enumerate}
\item We know that $S\simeq \mathbb{D}/K$, and that we can obtain a fundamental domain $F$ for $K$ joining $k$ regular $N$-gons of angle $2\pi/3$, where $N:=\frac{6(k-\chi)}{k} $. There are different  possible combinatorial configurations  for the relative position of the $k$ polygons, so we choose one to start with. Note that 
  we may need to change to a different configuration if  the process is not successful in the end.
  \item For the chosen fundamental domain, we compute the set $\mathcal{P}$ of all the possible side-pairing transformations between all the possible pairs of edges of the polygons, admitting orientation reversing transformations if we want to construct a non-orientable surface. The set  $\mathcal{P}$ is usually very large.
\item We compute a set as large as possible of admissible distances associated to a tessellation by regular $N$-gons of angle $2\pi /3$.
\item We choose a starting pair of transformations $p_1, p_2 \in  \mathcal{P}$, and compute the set $\mathcal{C}$ of points of $F$ that show certain  admissible displacements under $p_1$ and $p_2$, which will be our \emph{candidates}. By what has been said above, there are finitely many candidate points for every choice of  $F, p_1$ and $p_2$, and there are finitely many  choices of such a triple. Again, if our construction does not produce the surface we look for with this choice of $p_1, p_2$, we change to a different starting pair.
\item We compute the displacement of every point  $c \in \mathcal{C}$ under all the transformations in $ \mathcal{P}$, in order to determine the subset $ \mathcal{P}_c \subset  \mathcal{P}$ consisting of those transformations moving $c$ an admissible distance.
\item We check if there are enough transformations in $ \mathcal{P}_c $ in order to construct a full set of side-pairing transformations between pairs of edges of $F$ such that every vertex cycle is surrounded by a full $2\pi $ angle. If this is the case, the group $K$ generated by these transformations uniformizes a $k$-extremal surface of genus $g$, with an obvious extremal $k$-packing formed by the discs inscribed to the $k$-polygons that form $F$. Additionally,  the point $c$ may be the center of one of the discs of a second \emph{hidden} $k$-extremal disc.
\item We determine if there exists another configuration $F'$ of $k$ regular $N$-polygons which is also a fundamental domain for $K$, one of the polygon centers being the candidate $c$. If this is the case,  a second extremal $k$-packing  exists in $S=\mathbb{D}/K$. In practice this is a difficult step unless we can show the existence of an automorphism of $S$ moving the obvious $k$-packing to the hidden one.
\end{enumerate}

\section{An explicit example: the case $N=14$.}

 We illustrate now our strategy  for the case $N=14$, and recall our study can be easily adapted with minor modifications to the seven relevant values of $N$, namely $N=7,8,9,10,11,14$ and $16$.
 
 \
 
The primitive $N=14$ example we are looking for is a $3$-extremal compact non-orientable surface $S$ of genus $6$, which has as fundamental region the connected union $F$ of three regular $14$-gons with angle $\frac{2\pi}{3}$.

\begin{enumerate}
\item We begin with the region $F$ in Figure \ref{14-1}, consisting of three regular $14$-gons with angle $\frac{2\pi}{3}$. We label the edges of each polygon counterclockwise from $0$ to $13$.  There is no guarantee about being able to construct our surface starting with the domain $F$: if the rest of the process would fail, we would have to come back to this point and start again with a different shaped $F$.

\begin{figure}[!htbp]
\begin{center}
\includegraphics[width=0.5\textwidth]{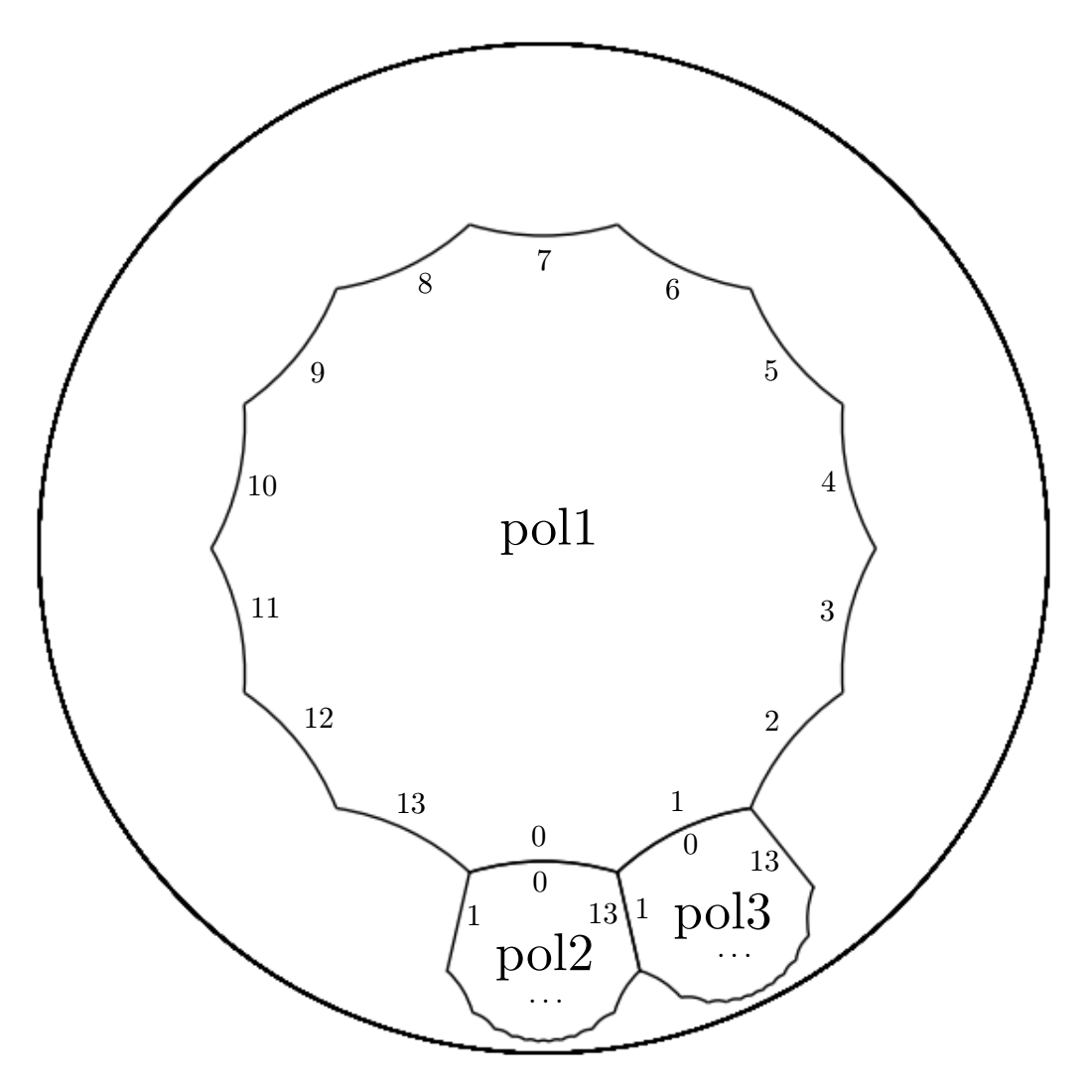} 
\caption{The starting fundamental region $F$.}\label{14-1}
\end{center}
\end{figure}

\item We compute all the conformal and anticonformal transformations that can be used to identify edges of $F$, collecting all this information in a list that we  call $L$ (see \cite{webpage} for the computer code and the results). The elements of $L$ are encoded with information like 
$$[\text{'hyperbolic'},\ \text{'de pol1 a pol3'},\ 8,\ 11]$$ 
this example meaning a conformal transformation which sends the inner triangle based at the edge $8$ of polygon $1$ to the outer triangle based at the edge $11$ of polygon $3$; see Figure \ref{14-2}.

\begin{figure}[!htbp]
\begin{center}
\includegraphics[width=0.5\textwidth]{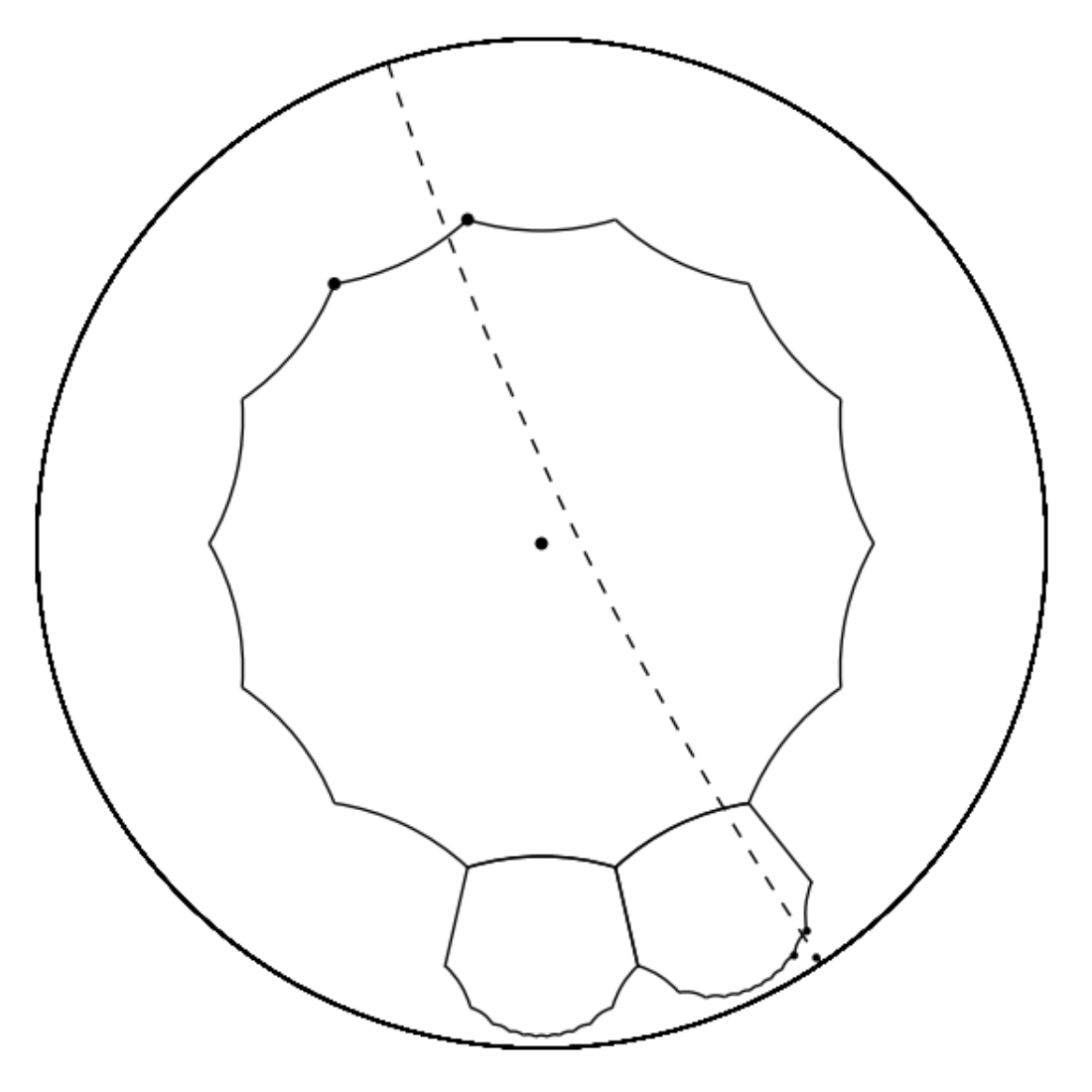} 
\caption{Example of a transformation $\gamma$  in $L$, together with its axis. Here $\gamma$ sends the edge 8 of pol1 to the edge 11 of pol3.}\label{14-2}
\end{center}
\end{figure}

\item We compute a set $D$ of admissible distances provided by the transformations in $\mathcal{R}_5$ (recall the notation for $\mathcal{R}_5$ after the proof of Lemma \ref{le:dists}). See \cite{webpage} again to see the code employed and the data obtained.
\item We choose the following starting pair of side-pairing transformations:

\smallskip

\centerline{$p_1=$ ['orientation-reversing hyperbolic', 'de pol1 a pol1', 13, 10]}
\centerline{$p_2=$ ['hyperbolic', 'de pol1 a pol2', 2, 7]}

Now we compute the set of points showing an admissible displacement under $p_1$ and $p_2$; in Figure \ref{14-3} we show some of the corresponding bananas (the code is once more available at \cite{webpage}).  

\begin{figure}[!htbp]
\begin{center}
\includegraphics[width=0.8\textwidth]{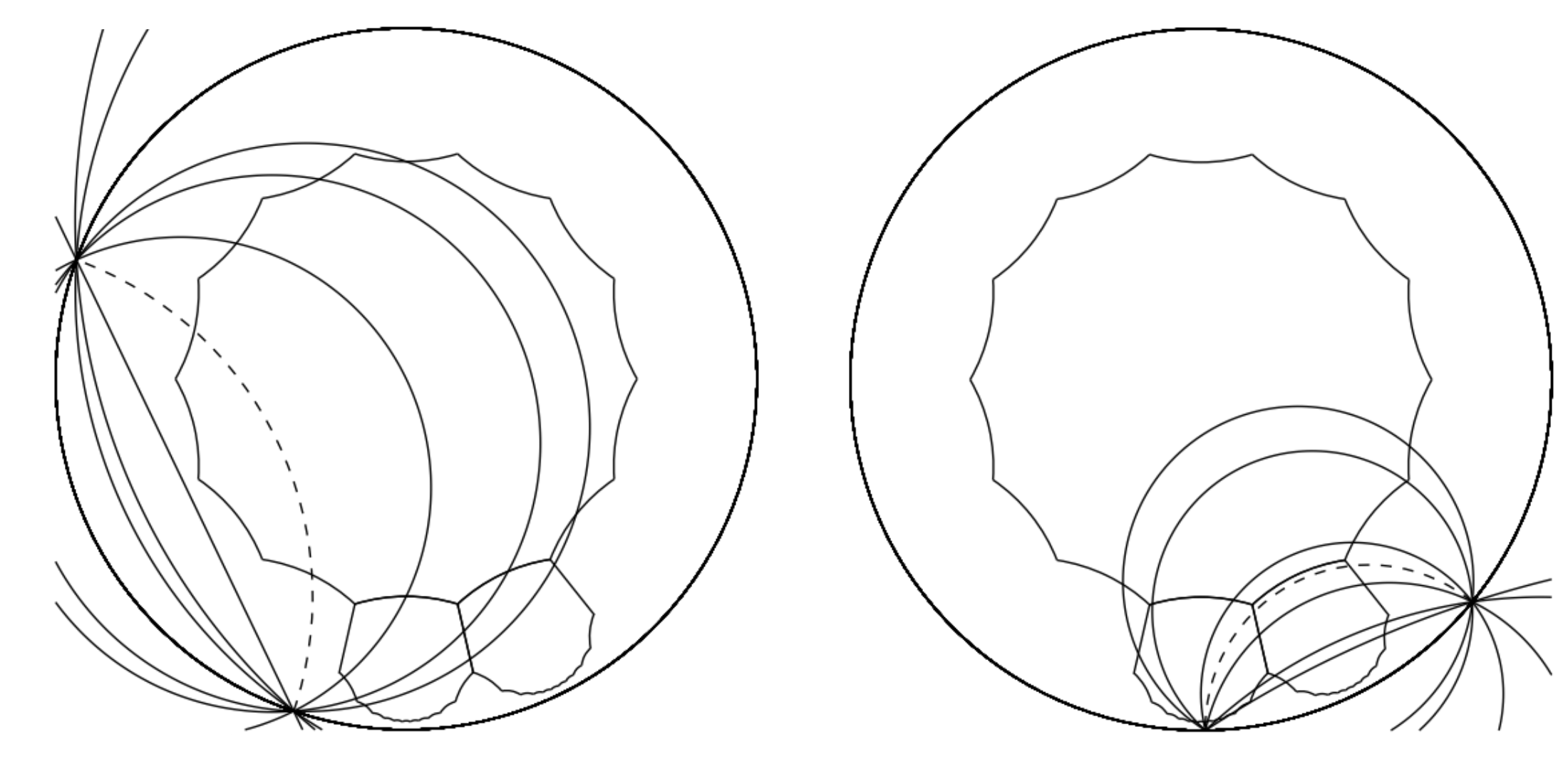} 
\caption{The starting set of  bananas corresponding  to $p_1$ (left) and $p_2$ (right). The dashed lines indicate the axes of these transformations.}\label{14-3}
\end{center}
\end{figure}

Figure \ref{14-4} shows some of the points in the set $\mathcal{C}$ of candidate points obtained as intersection of the bananas of $p_1$ and $p_2$.

\begin{figure}[!htbp]
\begin{center}
\includegraphics[width=0.6\textwidth]{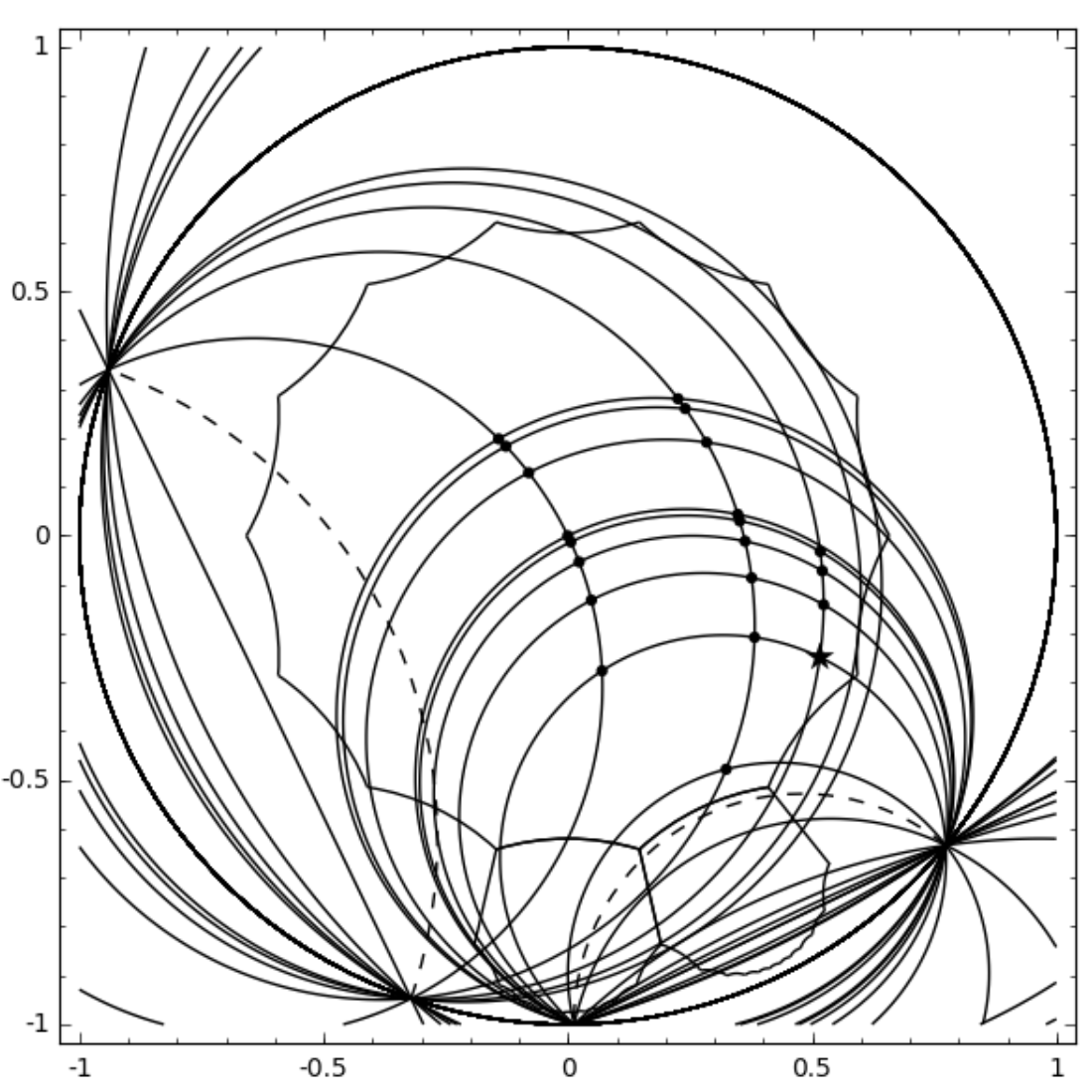} 
\caption{The construction of the set $\mathcal{C}$ of candidate points via  intersection of bananas of $p_1$ and $p_2$.}\label{14-4}
\end{center}
\end{figure}

\item For each of the points $c \in \mathcal{C}$, we collect in a list $L_c$ all the side pairings $p$ in $L$ which have a banana that passes through $c$ (computationally, we work up to an error $<10^{-4}$, see the code in \cite{webpage}).  For example, the list for the candidate closer to the bottom  of Figure \ref{14-4}, which has coordinates $c\simeq 0.324 - 0.478i$, is  

\smallskip

\centerline{['orientation-reversing hyperbolic', 'de pol1 a pol1', 10, 13]}
\centerline{ ['orientation-reversing hyperbolic', 'de pol1 a pol1', 13, 10]}
\centerline{['hyperbolic', 'de pol1 a pol2', 2, 7]}
\centerline{[orientation-reversing hyperbolic', 'de pol1 a pol2', 11, 1]}
\centerline{['hyperbolic', 'de pol2 a pol3', 8, 13]}

\item In the case of $c\simeq 0.324 - 0.478i$, we see for instance that $L_c$ does not contain a side pairing involving the edge 3 of pol1, so this point $c$ fails to be a good candidate. 

Only two of the points of $\mathcal{C}$ marked  in Figure \ref{14-4} satisfy the needed condition that $L_c$ contains enough transformations so that we can choose a side-pairing for every edge of $F$.
One is obviously the origin, and the other one is $c\simeq 0.516-0.248i$ marked $\star$ in Figure \ref{14-4}.

The list $L_c$ is huge, and contains several side pairings for every edge. For example, if we focus on edge $3$ of pol1, we find in $L_c$  the following side pairings:

\smallskip

\centerline{ ['hyperbolic', 'de pol1 a pol1', 3, 11]}  
\centerline{['orientation-reversing hyperbolic', 'de pol1 a pol1', 3, 8]}  
\centerline{['hyperbolic', 'de pol1 a pol2', 3, 10]}  
\centerline{['orientation-reversing hyperbolic', 'de pol1 a pol2', 3, 7]}  
\centerline{['hyperbolic', 'de pol1 a pol3', 3, 11]} 
\centerline{['orientation-reversing hyperbolic', 'de pol1 a pol3', 3, 8]}

Considering all the elements of $L_c$ that correspond to each edge, all we have to do is to choose, by direct inspection, a set of transformations that would pair all the edges of $F$ and generate a group uniformizing a compact surface. The precise set we find is:

\smallskip

\centerline{['orientation-reversing hyperbolic', 'de pol1 a pol1', 13, 10]}
\centerline{['hyperbolic', 'de pol1 a pol2', 2, 7]}
\centerline{['orientation-reversing hyperbolic', 'de pol2 a pol1', 1, 11]} 
\centerline{['hyperbolic', 'de pol3 a pol2', 13, 8]} 
\centerline{['hyperbolic', 'de pol1 a pol1', 12, 7]} 
\centerline{['orientation-reversing hyperbolic', 'de pol1 a pol2', 8, 2]} 
\centerline{['hyperbolic', 'de pol1 a pol2', 5, 3]}
\centerline{['orientation-reversing hyperbolic', 'de pol1 a pol1', 9, 6]}
\centerline{['hyperbolic', 'de pol1 a pol3', 4, 5]}
\centerline{['hyperbolic', 'de pol1 a pol2', 3, 10]} 
\centerline{['hyperbolic', 'de pol2 a pol3', 9, 6]} 
\centerline{['hyperbolic', 'de pol3 a pol3', 12, 7]} 
\centerline{['hyperbolic', 'de pol2 a pol2', 6, 11]} 
\centerline{['hyperbolic', 'de pol2 a pol3', 4, 4]} 
\centerline{['hyperbolic', 'de pol3 a pol3', 3, 11]} 
\centerline{['hyperbolic', 'de pol2 a pol3', 5, 10]} 
\centerline{['hyperbolic', 'de pol2 a pol3', 12, 9]} 
\centerline{['hyperbolic', 'de pol3 a pol3', 2, 8]}

\item We have constructed a surface $S=\mathbb{D}/K$ that has a point $P\simeq [0.516-0.248i]_K$ that may play a special role, as it satisfies the metric requirements that a center of one of the three discs of an extremal $3$-packing should satisfy. Showing that this second $3$-packing actually exists could be quite difficult, but before trying anything more sophisticated it is worth checking if, as it happens in the classical $k=1$ case (cf. \cite{Girondo_Gonzalez-Diez_2002_2}), an automorphism of $S$ is sending the original $3$-packing somewhere else.

\begin{figure}[!htbp]
\begin{center}
\includegraphics[width=0.5\textwidth]{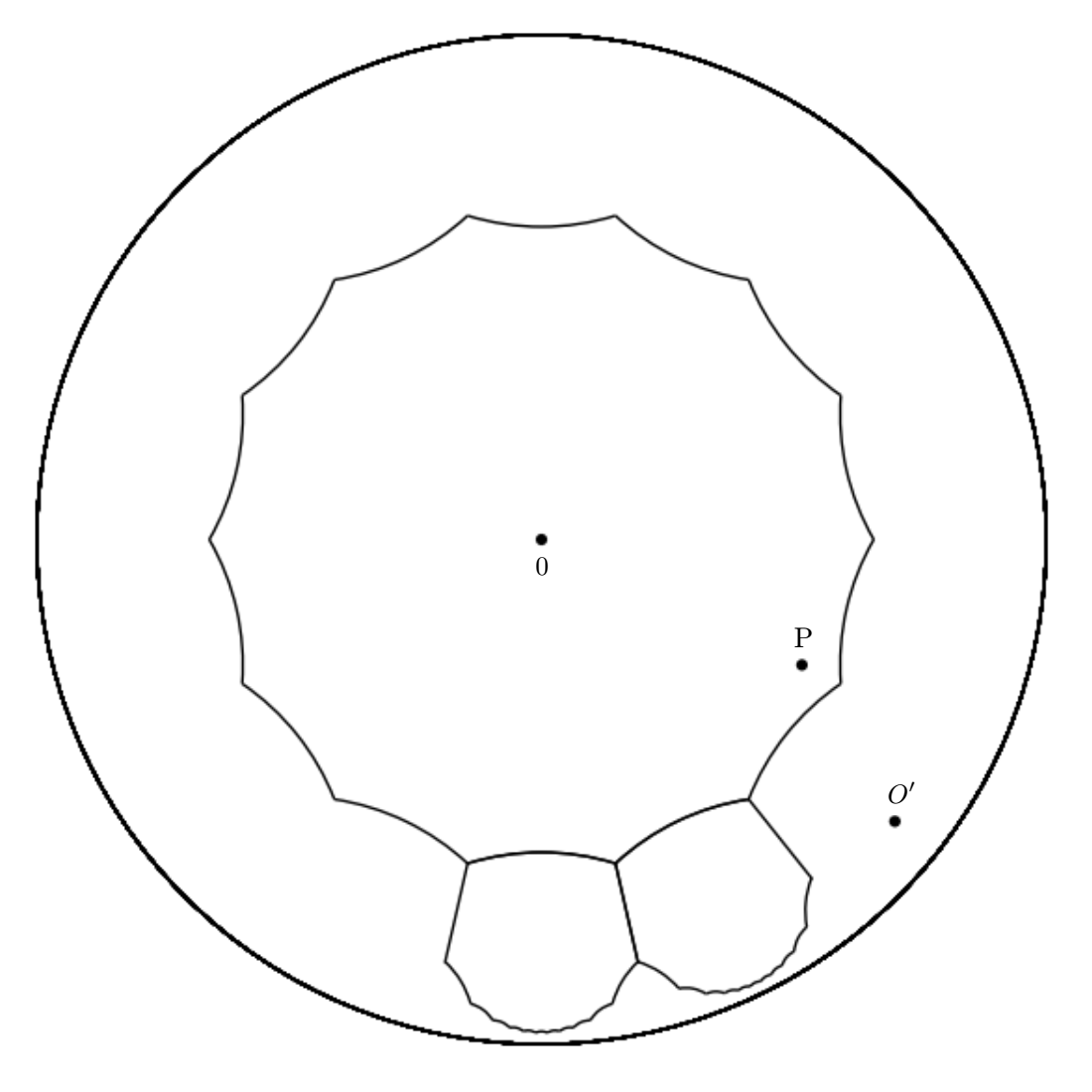} 
\caption{Our example for $N=14$.}\label{14-5}
\end{center}
\end{figure}

This is exactly what happens also in this case: if we denote $O':=(ca)^2b(0)$ we can show, after some computational work, that the elliptic element $\tau$ of order 2 fixing the midpoint between $P$ and $O'$ normalizes the group $K$, thus it induces an automorphism of $S$. The way how we show it is by conjugating every side-pairing of $F$ generating $K$ and checking that the resulting hyperbolic isometry also belongs  to $K$ (see \cite{webpage} for the data).

\

In order to have a nicer image of the surface, we replace pol2 by the (equivalent) polygon centered at $O'$ (see the left side of Figure \ref{14-8}). Even better, if we move the domain to be centered in the disc, we obtain our final picture (at the right side of Figure \ref{14-8}), where the side-pairing generators of the uniformizing group are also indicated. We recall that an order two automorphism transposes $P$ and $O'$.

\begin{figure}[!htbp]
\begin{center}
\includegraphics[width=1.0\textwidth]{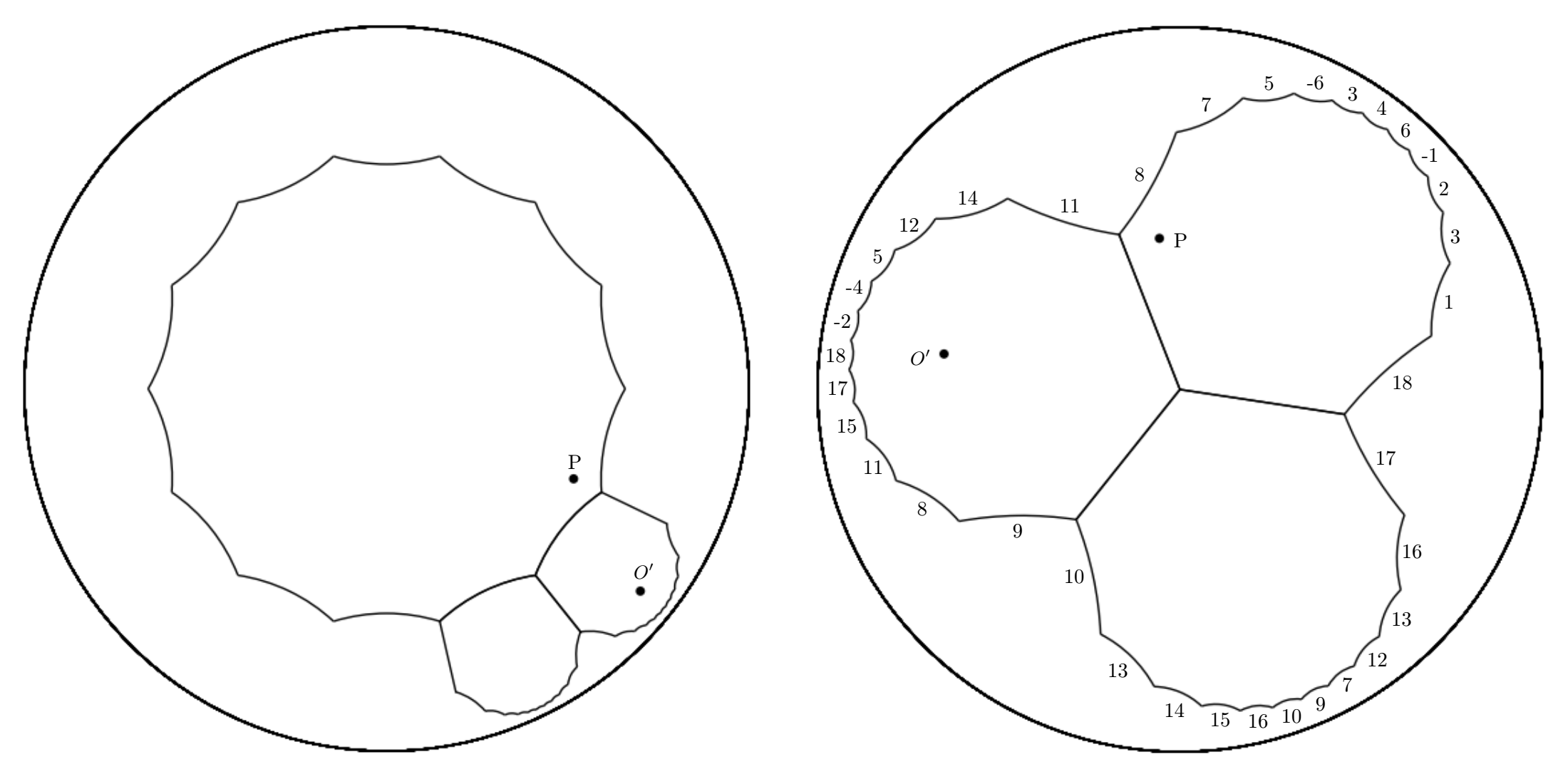} 
\caption{Two fundamental domains for the group uniformizing the compact non-orientable $3$-extremal surface of genus 6 constructed. The labels on the right side describe the side-pairing transformations.}\label{14-8}
\end{center}
\end{figure}

\end{enumerate}

\bibliography{references}  

\def\cprime{$'$}
\begin{thebibliography}{GGD02}

\bibitem[Bav96]{Bavard_1996}
Christophe Bavard.
\newblock Disques extr\'emaux et surfaces modulaires.
\newblock {\em Ann. Fac. Sci. Toulouse Math. (6)}, 5(2):191--202, 1996.

\bibitem[Bea83]{Beardon_1983}
A.~F. Beardon.
\newblock {\em The Geometry of Discrete Groups}, volume~91 of {\em Graduate
  Texts in Mathematics}.
\newblock Springer-Verlag, New York, 1983.

\bibitem[EEK82]{Edmonds_Ewing_Kulkarni_1982}
Allan~L. Edmonds, John~H. Ewing, and Ravi~S. Kulkarni.
\newblock Torsion free subgroups of {F}uchsian groups and tessellations of
  surfaces.
\newblock {\em Invent. Math.}, 69(3):331--346, 1982.

\bibitem[GGD99]{Girondo_Gonzalez-Diez_1999}
Ernesto Girondo and Gabino Gonz{\'a}lez-Diez.
\newblock On extremal discs inside compact hyperbolic surfaces.
\newblock {\em C. R. Acad. Sci. Paris S\'er. I Math.}, 329(1):57--60, 1999.

\bibitem[GGD02]{Girondo_Gonzalez-Diez_2002_2}
Ernesto Girondo and Gabino Gonz{\'a}lez-Diez.
\newblock Genus two extremal surfaces: extremal discs, isometries and
  {W}eierstrass points.
\newblock {\em Israel J. Math.}, 132:221--238, 2002.

\bibitem[Gir18]{Girondo_2018}
Ernesto Girondo.
\newblock Extremal disc packings in compact hyperbolic surfaces.
\newblock {\em Rev. Mat. Complut.}, 31(2):467--478, 2018.

\bibitem[Gir19]{webpage}
Ernesto Girondo.
\newblock Personal webpage at universidad autonoma de madrid.
\newblock \url{http://matematicas.uam.es/~ernesto.girondo/}, may 2019.

\bibitem[GN07]{Girondo_Nakamura_2007}
Ernesto Girondo and Gou Nakamura.
\newblock Compact non-orientable hyperbolic surfaces with an extremal metric
  disc.
\newblock {\em Conform. Geom. Dyn.}, 11:29--43 (electronic), 2007.

\bibitem[GR19]{preprint}
Ernesto {Girondo} and Cristian {Reyes}.
\newblock {Extremal k-packings in compact non-orientable surfaces}.
\newblock {\em arXiv e-prints}, page arXiv:1905.13295, May 2019.

\bibitem[Nak09]{Nakamura_2009}
Gou Nakamura.
\newblock Compact non-orientable surfaces of genus 4 with extremal metric
  discs.
\newblock {\em Conform. Geom. Dyn.}, 13:124--135, 2009.

\bibitem[Nak12]{Nakamura_2012}
Gou Nakamura.
\newblock Compact non-orientable surfaces of genus 5 with extremal metric
  discs.
\newblock {\em Glasg. Math. J.}, 54(2):273--281, 2012.

\bibitem[Nak13]{Nakamura_2013}
Gou Nakamura.
\newblock Compact {K}lein surfaces of genus 5 with a unique extremal disc.
\newblock {\em Conform. Geom. Dyn.}, 17:39--46, 2013.

\bibitem[Nak16]{Nakamura_2016}
Gou Nakamura.
\newblock Compact non-orientable surfaces of genus 6 with extremal metric
  discs.
\newblock {\em Conform. Geom. Dyn.}, 20:218--234, 2016.

\bibitem[Tak77]{Takeuchi_1977_a}
Kisao Takeuchi.
\newblock Arithmetic triangle groups.
\newblock {\em J. Math. Soc. Japan}, 29(1):91--106, 1977.

\end{thebibliography}
\bibliographystyle{alpha}

\end{document}